\documentclass[11pt]{amsart}

\usepackage[margin=1.5in]{geometry} 
\usepackage{amsmath,amsthm,amssymb,bbm,bm}
\usepackage{graphicx}
\usepackage{hyperref}
\usepackage{amssymb,amsfonts, bm}
\usepackage[all,arc]{xy}
\usepackage{enumerate}
\usepackage{mathrsfs}
\usepackage{tikz}

\newcommand{\R}{\mathbb{R}}

\newcommand{\I}{\mathbbm{1}}
\newcommand{\Var}{\mathrm{Var}}
\newcommand{\cov}{\mathrm{cov}}
\newcommand{\Poi}{\mathrm{Poisson}}
\newcommand{\CM}{\mathrm{CM}}

\newcommand{\Address}{{
\bigskip
\footnotesize

\textsc{Department of Mathematics, University of Southern California, Los Angeles, CA, 90089}\par\nopagebreak
\textit{E-mail address}: \texttt{paguyo@usc.edu}
}}

\def\bal#1\eal{\begin{align*}#1\end{align*}}

\newtheorem{theorem}{Theorem}[section]
\newtheorem{lemma}[theorem]{Lemma}
\newtheorem{proposition}[theorem]{Proposition}

\title{Convergence rates of limit theorems in random chord diagrams} %%% V2 %%%
\author{J. E. Paguyo } 
\date{}


\begin{document}

%%% ABSTRACT %%%

\begin{abstract}
We study the asymptotic distributions of the number of crossings and the number of simple chords in a random chord diagram. 
Using size-bias coupling and Stein's method, we obtain bounds on the Kolmogorov distance between the distribution of the number of crossings and a standard normal random variable, 
and on the total variation distance between the distribution of the number of simple chords and a Poisson random variable. 
As an application, we provide explicit error bounds on the number of chord diagrams containing no simple chords. 
\end{abstract}

\maketitle

%%% INTRODUCTION %%%

\section{Introduction}

A {\em chord diagram} of size $n$ is a matching of $2n$ points on a circle, labeled in clockwise order, with each matching corresponding to a {\em chord}. 
Alternatively, we can represent it as a {\em linearized chord diagram} by placing $2n$ points on a line in increasing order and connecting pairs by arcs. See Figure \ref{Fig1ChordDiagram} for an example. 
There are $(2n-1)!!$ chord diagrams of size $n$. 

A {\em matching}, or {\em fixed-point free involution}, is a permutation $\pi \in S_{2n}$ such that $\pi^2 = 1$ and $\pi(i) \neq i$ for all $i \in \{1,\ldots,2n\}$. 
Thus we may represent a chord diagram, $C_n$, by its corresponding matching, $\pi$, and write $C_n = C_n(\pi) = \pi$. 
In the remainder of the paper, we go back and forth between $C_n$ and $\pi$ to denote a chord diagram whenever one representation is more useful in context than the other. 
For example, the chord diagram in Figure \ref{Fig1ChordDiagram} can be written as a matching in cycle notation as $\pi = (1 \, 8)(2 \, 9)(3 \, 4)(5 \, 7)(6 \, 10)(11 \, 12)$. 

\begin{figure}
\scalebox{0.6}{
\begin{tikzpicture}
	\def \n {20}
        \def \radius {2.5}
        \draw circle(\radius)
              foreach \x in{1,...,12}{
                  (-360/12*\x - 90: -\radius) node[fill, circle, inner sep=1.5pt, minimum size = 0.07cm]{}
                  node[anchor=-360/12*\x - 90] {\small $\x$}
              };
         \draw[<->] (-360/12*1 - 90: -\radius) -- (-360/12*8 - 90: -\radius)
         		(-360/12*2 - 90: -\radius) -- (-360/12*9 - 90: -\radius)
			(-360/12*3 - 90: -\radius) -- (-360/12*4 - 90: -\radius)
			(-360/12*5 - 90: -\radius) -- (-360/12*7 - 90: -\radius)
			(-360/12*6 - 90: -\radius) -- (-360/12*10 - 90: -\radius)
			(-360/12*11 - 90: -\radius) -- (-360/12*12 - 90: -\radius);
\end{tikzpicture}
\begin{tikzpicture}[out = 45, in = 135]
	\foreach \x in {1,...,12}
		\fill (\x,0) circle (0.07cm) node[anchor = north] {\small $\x$};
	\draw (1,0) to (8,0)
		 (2,0) to (9,0)
		 (3,0) to (4,0)
		 (5,0) to (7,0)
		 (6,0) to (10,0)
		 (11,0) to (12,0);
\end{tikzpicture}
}
\caption{A chord diagram and its linearized version.}
\label{Fig1ChordDiagram}
\end{figure}
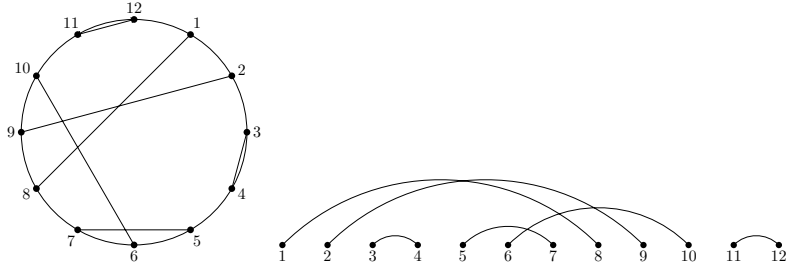

We now define terminology and fix some notation that will be used throughout the paper. For a positive integer $n$, let $[n] = \{1,\ldots,n\}$. We denote $(2n-1)!!$ to be the product of odd integers from $1$ to $2n-1$. 
If there exists positive constants $c$ and $n_0$ such that $a_n \leq cb_n$ for all $n \geq n_0$, then we write $a_n = O(b_n)$.
The chord connecting any two points $x$ and $y$ is denoted by $(x,y)$. 

A {\em connected} chord diagram is a diagram where no set of chords can be separated from the remaining chords by a line. A {\em component} is a maximal connected subdiagram. 
The {\em root component} is the component that contains the point labeled $1$. 
A {\em crossing} is a quadruple $(i,j,k,\ell)$, with $i < j < k < \ell$, such that $(i,k)$ and $(j,\ell)$ are chords. 
On the other hand, a {\em nesting} is a quadruple $(i,j,k,\ell)$, with $i < j < k < \ell$, such that $(i,\ell)$ and $(j,k)$ are chords. 
A {\em simple chord} is a chord that connects two consecutive endpoints; that is, a simple chord is of the form $(i,i+1)$, for $i \in [2n]$, where addition is understood to be modulo $2n$. 
For example, the chord diagram in Figure \ref{Fig1ChordDiagram} has $4$ crossings, $4$ nestings, $3$ components, and $2$ simple chords. 

The enumeration of chord diagrams was initiated by Touchard in \cite{Tou52}, where he found a bivariate generating function for $T_{n,k}$, the number of chord diagrams of size $n$ with exactly $k$ crossings. 
In particular, this gives $T_{n,0} = \frac{1}{n+1}\binom{2n}{n}$, the classical result that the number of nonintersecting chord diagrams of size $n$ is given by the $n$th Catalan number. 
Riordan \cite{Rio75} extended Touchard's results by finding an explicit formula for $T_{n,k}$ in the form of an alternating sum. 
This allows for the computation of $T_{n,k}$ for moderately sized $n$ and $k$, however obtaining asymptotics on $T_{n,k}$ as $n,k \to \infty$ remains difficult.

Probabilistic questions have also been studied. Using recurrence relations, Stein and Everett \cite{SE78} showed that the probability a random chord diagram is connected approaches $1/e$ as $n \to \infty$. 
Flajolet and Noy \cite{FN00} used generating functions and analytic combinatorics to find the asymptotic distributions of the number of components, the size of the largest component, and the number of crossings. 
Moreover they showed that a random chord diagram is {\em monolithic} with high probability, meaning it consists of one large connected component and some isolated chords. 
In \cite{CDDSY07}, Chen et al showed that the crossing numbers and nesting numbers of matchings have a symmetric joint distribution. 
More recently, Acan \cite{Aca17} extended Flajolet and Noy's results about the components of a random chord diagram in several directions, 
and Acan and Pittel \cite{AP17} discovered the emergence of a giant component in the intersection graph of a random chord diagram, under some conditions on the number of crossings. 

Chord diagrams have found applications in a wide variety of fields such as topology, random graph theory, biology, quantum field theory, and free probability \cite{APRW13, BR00, BR04, HZ86, Kon93, LN11, Mah20, NS06}.

%%% MAIN RESULTS %%%

\subsection{Main Results}

Let $\mu$ and $\nu$ be probability distributions. The {\em total variation distance} between $\mu$ and $\nu$ is
\bal
d_{TV}(\mu, \nu) := \sup_{A \subseteq \Omega} |\mu(A) - \nu(A)|
\eal
where $\Omega$ is a measurable space. The {\em Kolmogorov distance} between $\mu$ and $\nu$ is 
\bal
d_K(\mu, \nu) := \sup_{x\in \R} |\mu(-\infty, x] - \nu(-\infty, x]|.
\eal
If $X$ and $Y$ are random variables with distributions $\mu$ and $\nu$, respectively, then we write $d_K(X,Y)$ to denote the Kolmogorov distance between the distributions of $X$ and $Y$, and similarly for $d_{TV}$. 

In \cite{FN00}, Flajolet and Noy used generating functions to prove a central limit theorem for the number of crossings in a random chord diagram. Another proof using weighted dependency graphs was given by Feray in \cite{Fer18}. 
The main difficulty with the crossing statistic is that occurrences of crossings in disjoint sets of indices are {\em not} independent events. 

Our main result gives a rate of convergence for the asymptotic normality of the number of crossings by providing an upper bound of order $n^{-1/2}$ on the 
Kolmogorov distance between the standardized number of crossings and a standard normal random variable. 

\begin{theorem} \label{crossingStein}
Let $X_n$ be the number of crossings in a random chord diagram of size $n \geq 2$. Let $W_n = \frac{X_n - \mu_n}{\sigma_n}$, where $\mu_n = EX_n$ and $\sigma_n^2 = \Var(X_n)$. Then 
\bal
d_K(W_n,Z) = O\left( \frac{1}{\sqrt{n}} \right),
\eal
where $Z$ is a standard normal random variable. 
\end{theorem}

The proof uses size-bias coupling and Stein's method. A major advantage of using Stein's method in proving limit theorems is that it gives explicit error bounds on the distributional convergence. 

Size-bias coupling and Stein's method have previously been used to prove central limit theorems, for example, by Conger and Viswanath \cite{CV07} for the number of inversions and descents of permutations of multisets, 
Betken \cite{Bet19} for the number of isolated vertices in preferential attachment random graphs, and He \cite{He22} for the number of descents in a Mallows distributed permutation and its inverse. 

We also use size-bias coupling and Stein's method to obtain an upper bound on the total variation distance between the number of simple chords and a Poisson random variable. 

\begin{theorem}\label{simplechordsStein}
Let $S_n$ be the number of simple chords in a random chord diagram of size $n$ and let $Y$ be a $\Poi(1)$ random variable. Then 
\bal
d_{TV}(S_n, Y) = O\left(\frac{1}{n}\right). 
\eal
\end{theorem}

As an application of Theorem \ref{simplechordsStein}, we obtain absolute error bounds on the number of chord diagrams of size $n$ which contain no simple chords. 

\begin{theorem}\label{simplechordfreechorddiagrams}
Let $s(n)$ be the number of chord diagrams of size $n$ with no simple chords. Then 
\bal
\frac{(2n-1)!!}{e}\left(e^{-\frac{1}{2n-1}} - \frac{10}{n} \right) \leq s(n) \leq \frac{(2n-1)!!}{e}\left(e^{-\frac{1}{2n-1}} + \frac{10}{n} \right).
\eal 
\end{theorem}

We remark that Stein's method via size-bias coupling has been successfully used to prove Poisson limit theorems. 
Some examples are in Angel, van der Hofstad, and Holmgren \cite{AHH19} for the number self-loops and multiple edges in the configuration model, 
Arratia and DeSalvo \cite{AD17} for completely effective error bounds for Stirling numbers of the first and second kinds, 
Goldstein and Reinert \cite{GR06} for the size of the intersection of random subsets of $[n]$, 
and Holmgren and Janson \cite{HJ15} for sums of functions of subtrees of random binary search trees and random recursive trees. 

%%% OUTLINE %%%

\subsection{Outline}

The paper is organized as follows. In Section \ref{Crossings} we give a brief overview of size-bias coupling and Stein's method. 
We then construct a size-bias coupling for the number of crossings and combine this with a size-bias coupling version of Stein's method to prove Theorem \ref{crossingStein}.

In Section \ref{SimpleChords} we use size-bias coupling and Stein's method for Poisson approximation to prove Theorem \ref{simplechordsStein}. We then use this to prove Theorem \ref{simplechordfreechorddiagrams}. 

We conclude the paper with final remarks and some open problems.

%%%%%%%%%%   NUMBER OF CROSSINGS   %%%%%%%%%%

\section{Crossings}\label{Crossings}

A crossing in a chord diagram, $C_n$, is a quadruple $(a,b,c,d)$ with $a < b < c < d$ such that $(a,c)$ and $(b,d)$ are chords in $C_n$. 
Let $X_n$ be the number of crossings in $C_n$. Let $N = \binom{2n}{4}$ and order the set of quadruples $(a,b,c,d)$, with $a<b<c<d$, in some arbitrary but fixed labeling. Let $\I_{(i,j)}$ be the indicator random variable that chord $(i,j)$ is in $C_n$. 
For a quadruple $k = (a,b,c,d)$, let $Y_k = \I_{(a,c)}\I_{(b,d)}$, so that $Y_k$ is the indicator random variable that the quadruple $k$ forms a crossing in $C_n$. Then we can write $X_n$ as a sum of indicators as $X_n = \sum_{k=1}^N Y_k$.

The next lemma is useful for the remainder of the paper. We omit the straightforward proof. 

\begin{lemma}\label{edgeprobability}
Given distinct $i_1,\ldots,i_k,j_1,\ldots, j_k \in [2n]$, the probability that a random chord diagram, $C_n$, contains the set of chords $\{(i_m,j_m)\}_{1 \leq m \leq k}$ is
\bal
P(\{(i_m,j_m)\}_{1 \leq m \leq k} \in C_n) = \frac{(2n-2k-1)!!}{(2n-1)!!} = \frac{1}{(2n-1)\dotsb (2n-2k+1)}.
\eal
\end{lemma}

\subsection{Mean and Variance}

Let $\mu_n = E(X_n)$ and $\sigma_n^2 = \Var(X_n)$ be the mean and variance of $X_n$. The exact values for $\mu$ and $\sigma$ were stated without proof in \cite{Rio75}. 
It was subsequently proven via generating functions in \cite{FN00}, where moments of any fixed order was determined. 
A more direct proof using the method of indicators, with some computer assistance for the variance computation, was given in \cite{Fer18}.

\begin{lemma}[\cite{FN00}, Theorem 3]\label{crossingsmeanvar}
Let $X_n$ be the number of crossings in a random chord diagram of size $n$. Then
\bal
\mu_n = \frac{n(n-1)}{6} \quad \text{and} \quad \sigma_n^2 = \frac{n(n-1)(n+3)}{45}.
\eal
\end{lemma}

\subsection{Size-Bias Coupling and Stein's Method} \label{sizebiascouplingSteinmethod}

Let $X$ be a non-negative integer-valued random variable with finite mean. Then $X^s$ has the {\em size-bias distribution} of $X$ if
\bal
P(X^s = x) = \frac{xP(X = x)}{EX}.
\eal
A {\em size-bias coupling} is a pair $(X, X^s)$ of random variables defined on the same probability space such that $X^s$ has the size-bias distribution of $X$. 
To prove our main theorem, we will use the following size-bias coupling version of Stein's method. 

\begin{theorem}\label{sizebiasStein}
Let $X \geq 0$ be a random variable such that $\mu = EX < \infty$ and $\Var(X) = \sigma^2$. Let $X^s$ have the size-bias distribution with respect to $X$. If $W = \frac{X - \mu}{\sigma}$ and $Z \sim N(0,1)$, then 
\bal
d_K(W,Z) \leq \frac{2\mu}{\sigma^2} \sqrt{\Var(E[X^s - X \mid X])} + \frac{8\mu|X^s - X|^2}{\sigma^3}.
\eal
\end{theorem}

\begin{proof}
This follows by applying Construction 3A to Corollary 2.6 in \cite{CR10}. 
\end{proof}

Observe that this implies a central limit theorem if both error terms on the right hand side go to zero. 

We now turn to the construction of a size-biased version of the number of crossings $X_n$. We use the following recipe provided in \cite{Ros11} for coupling a random variable $X$ with its size-bias version $X^s$. 

Let $X = \sum_{i=1}^n X_i$ where $X_i \geq 0$ and $\mu_i = EX_i$. 
\begin{enumerate}
\item For each $i \in [n]$, let $X_i^s$ have the size-bias distribution of $X_i$ independent of $(X_k)_{k \neq i}$ and $(X_k^s)_{k \neq i}$. 
Given $X_i^s = x$, define the vector $(X_k^{(i)})_{k \neq i}$ to have the distribution of $(X_k)_{k \neq i}$ conditional on $X_i = x$. 
\item Choose a random summand $X_I$, where the index $I$ is chosen, independent of all else, with probability $P(I = i) = \mu_i/\mu$, where $\mu = EX$. 
\item Define $X^s = \sum_{k \neq I} X_k^{(I)} + X_I^s$. 
\end{enumerate}

If $X^s$ is constructed by Items (1)-(3) above, then $X^s$ has the size-bias distribution of $X$. As a special case, note that if the $X_i$ is a zero-one random variable, then its size-bias distribution is $X_i^s = 1$. 
We summarize this as the following proposition. 

\begin{proposition}[\cite{Ros11}, Corollary 3.24]\label{sizebiascouplingconstruction}
Let $X_1,\ldots, X_n$ be zero-one random variables and let $p_i := P(X_i = 1)$. For each $i \in [n]$, let $(X_k^{(i)})$ have the distribution of $(X_k)_{k \neq i}$ conditional on $X_i = 1$. If $X = \sum_{i=1}^n X_i$, 
$\mu = EX$, and $I$ is chosen independent of all else with $P(I = i) = p_i/\mu$, then $X^s = \sum_{k \neq I} X_k^{(I)} + 1$ has the size-bias distribution of $X$. 
\end{proposition}

\subsection{Construction} \label{construction}

Following steps (1)-(3) above, we construct a random variable $X_n^s$ having the size-bias distribution with respect to the number of crossings $X_n$. 

Fix a chord diagram $C_n$ and let $X_n$ be the number of crossings. Pick an index $I = (a,b,c,d)$ uniformly at random from $[N]$, independent of all else. We form a new chord diagram $C_n^s$ as follows. If $Y_I = 1$, set $C_n^s = C_n$. 
Otherwise, there is no crossing at index $I$.  
Delete all edges containing $a,b,c,d$ as endpoints. Then form the edges $(a,c)$ and $(b,d)$, so that there is a crossing at index $I$. 
Finally form the edges $(\pi(a), \pi(c))$ and $(\pi(b), \pi(d))$. 
Let $C_n^s$ be the resulting chord diagram. 

Set $X_n^s = \sum_{k\neq I} Y_k^{(I)} + 1$, where $Y_k^{(I)}$ is the indicator random variable that the new chord diagram $C_n^s$ has a crossing at index $k$. 
Observe that $C_n^s$ is a uniformly random chord diagram, conditioned on $Y_I = 1$, since the chord diagram formed by the points $[N] \setminus I$ is a uniformly random chord diagram on $2n-4$ points. 
Thus $(Y_k^{(I)})$ has the desired distribution of $(Y_k)_{k \neq I}$ conditional on $Y_I = 1$. By Proposition \ref{sizebiascouplingconstruction}, $X_n^s$ has the size-bias distribution of $X_n$. 
We record this as the following proposition. 

\begin{proposition}\label{crossingsizebiascoupling}
Let $X_n$ be the number of crossings in a random chord diagram of size $n$. Let $X_n^s$ be constructed as above. Then $X_n^s$ has the size-bias distribution of $X_n$. 
\end{proposition} 

\subsection{Central Limit Theorem for the Number of Crossings}

In this section we use the size-bias coupling from the previous section combined with the size-bias version of Stein's method, Theorem \ref{sizebiasStein}, to prove Theorem \ref{crossingStein}. 
The main difficulty is in bounding the variance term of Theorem \ref{sizebiasStein}. The following lemma will be useful. 

\begin{lemma}[\cite{He22}, Lemma 5.1]\label{covariancebound}
Let $X$ and $Y$ be random variables with $|X| \leq C_1$ and $|Y| \leq C_2$. Let $A$ be some event such that conditional on $A$, $X$ and $Y$ are uncorrelated. Then
\bal
|\cov(X,Y)| \leq 4C_1C_2P(A^c).  
\eal
\end{lemma}

We begin by bounding the variance term in Theorem \ref{sizebiasStein}. 

\begin{lemma}\label{Steinvariancebound}
Let $X_n$ be the number of crossings in a random chord diagram, $\pi$, and let $X_n^s$ have the size-bias distribution with respect to $X_n$. Then
\bal
\Var(E[X_n^s - X_n \mid X_n]) = O(n).
\eal
\end{lemma}

\begin{proof}
First note that 
\bal
\Var(E[X_n^s - X_n \mid X_n]) \leq \Var(E[X_n^s - X_n \mid \pi]). 
\eal
Let $X_n^{(i)}$ denote $X_n^s$ conditioned to have a crossing at index $I = i$. Then 
\bal
E(X_n^s - X_n \mid \pi) &= \sum_{i=1}^N E(X_n^s - X_n \mid \pi, I = i)P(I = i) = \frac{1}{N} \sum_{i=1}^N (X_n^{(i)} - X_n),
\eal
where $X_n^{(i)} - X_n$ is the change in the number of crossings. 
Thus we can write the variance as
\bal
\Var(E[X_n^s - X_n \mid \pi]) = \frac{1}{N^2} \sum_{1 \leq i,j \leq N} \cov(X_n^{(i)} - X_n, X_n^{(j)} - X_n),
\eal 

By construction, $|X_n^{(i)} - X_n| \leq 4n$ since there are at most four new chords in $\pi^s$ and each chord either creates or destroys at most $n$ new crossings. 
We split up the sum according to whether the two sets of indices $i,j$ satisfy $|i \cap j| \neq 0$ or $|i \cap j| = 0$. 

Observe that the number of variance terms is at most $n^4$ and the number of covariance terms with $|i \cap j| \neq 0$ is at most $16n^7$. For these terms, it is enough to use the bound 
\bal
\cov(X_n^{(i)} - X_n, X_n^{(j)} - X_n) \leq 16n^2.
\eal
Thus the contribution of these variance and covariance terms to the sum is upper bounded by $16n^2(n^4 + 16n^7) \leq 272n^9$. 

It remains to bound covariance terms such that $|i \cap j| = 0$. There are at most $n^8$ such terms. Suppose $i = (a,b,c,d)$ and $j = (e,f,g,h)$. 
Define $i_\pi := (\pi(a), \pi(b), \pi(c), \pi(d))$ and $j_\pi := (\pi(e), \pi(f), \pi(g), \pi(h))$, the corresponding endpoints paired with the points in $i$ and $j$, respectively, 
where possibly $i \cap i_\pi \neq \emptyset$, $j \cap j_\pi \neq \emptyset$, $i_\pi \cap j \neq \emptyset$, or $j_\pi \cap i \neq \emptyset$. 

Let $A$ be the event given by the set of chord diagrams $\pi$ such that $|i_\pi \cap j| = |j_\pi \cap i| = 0$, which is the event where the neighbors of $i$ are disjoint from $j$ and the neighbors of $j$ are disjoint from $i$. 
Then conditional event $A$, $X_n^{(i)} - X_n$ and $X_n^{(j)} - X_n$ are independent. 

To see this, observe that under the size-bias construction of $\pi^{(i)}$ from $\pi$, the edges $(a,c)$, $(b,d)$, $(\pi(a), \pi(c))$, and $(\pi(b), \pi(d))$ are created, while the edges $(a, \pi(a))$, $(b,\pi(b))$, $(c, \pi(c))$, $(d, \pi(d))$ are destroyed. Define $B_i$ to be this set of edges which are created and destroyed edges under the size-bias construction of $\pi^{(i)}$ from $\pi$. 
Then the difference $X_n^{(i)} - X_n$ is only dependent on crossings which contain at least one edge from $B_i$. 
Similarly, define $B_j$ to be the set of edges which are created and destroyed under the size-bias construction of $\pi^{(j)}$ from $\pi$. Then the difference $X_n^{(j)} - X_n$ is only dependent on crossings which contain at least one edge from $B_j$. Finally note that the event $A$ implies that $B_i$ and $B_j$ are disjoint sets. 
It follows that conditional on $A$, $X_n^{(i)} - X_n$ and $X_n^{(j)} - X_n$ are independent. 

We now bound the probability that the complementary event $A^c$ occurs. 
Observe that $A^c$ is a subset of the union of events $\{\pi : \pi(k) = \ell\}$, where the union is taken over all $k \in \{a,b,c,d\}$ and $\ell \in \{e,f,g,h\}$. 
By the union bound and Lemma \ref{edgeprobability}, $P(A^c) \leq \frac{16}{2n-1}$. 

Thus by Lemma \ref{covariancebound}, the covariance terms such that $|i \cap j| = 0$ is bounded by
\bal
|\cov(X_n^{(i)} - X_n, X_n^{(j)} - X_n)| \leq 4(4n)^2\cdot \frac{16}{2n-1} \leq 1024n. 
\eal
Therefore the contribution of these covariance terms is $1024n^9$. 

Combining the covariance terms above gives
\bal
\Var(E[X_n^s - X_n \mid \pi]) \leq \frac{272n^9 + 1024n^9}{N^2} \leq (432)^2n,
\eal
which gives the desired $O(n)$ upper bound. 
\end{proof}

\begin{proof}[Proof of Theorem \ref{crossingStein}]
Let $(X_n, X_n^s)$ be the size-bias coupling from Section \ref{construction}, so that by Proposition \ref{crossingsizebiascoupling} $X_n^s$ has the size-bias distribution of $X_n$. 
By construction, we have that $|X_n^s - X_n| \leq 4n$. Moreover $n\geq 2$, we have that $\mu_n \leq \frac{n^2}{6}$ and $\sigma_n^2 \geq \frac{n^3}{45}$.

Therefore by Lemmas \ref{crossingsmeanvar}, \ref{Steinvariancebound} and Theorem \ref{sizebiasStein},
\bal
d_K(W_n,Z) &\leq \frac{2\mu_n}{\sigma_n^2} \sqrt{\Var(E[X_n^s - X_n \mid X_n])} + \frac{8\mu_n|X_n^s - X_n|^2}{\sigma_n^3}. \\
& \leq \frac{2n^2}{6} \cdot \frac{45}{n^3} \cdot 432n^{1/2} + \frac{8n^2}{6} \cdot \frac{45^{3/2}}{n^{9/2}} \cdot16n^2 \\
&\leq 12920n^{-1/2}. 
\eal
It follows that $d_K(W_n,Z) = O(n^{-1/2})$ as desired.
\end{proof}

Theorem \ref{crossingStein} shows that $d_K(W_n,Z) \to 0$ as $n \to \infty$. 
Therefore we recover the central limit theorem for the number of crossings, along with a rate of convergence. 

%%%%%%%%%%%%%%%%%%%%%%%%%%%%%%%%%%%%%%
%%%%%%%%%%   NUMBER OF SIMPLE CHORDS   %%%%%%%%%%

\section{Simple Chords}\label{SimpleChords}

Let $S_n$ be the number of simple chords in a random chord diagram of size $n$. We can write $S_n$ as a sum of indicators as $S_n = \sum_{k=1}^{2n} X_k$, where $X_k$ is the indicator random variable that $(k,k+1)$ is a simple chord in $C_n$. 

\subsection{Poisson Limit Theorem for Simple Chords}

Recall the relevant definitions and properties of size-bias distributions and size-bias couplings from Section \ref{Crossings}. 
We will use the following size-bias coupling version of Stein's method for Poisson approximation to prove the Poisson limit theorem for the number of simple chords. 

\begin{theorem}[\cite{Ros11}, Theorem 4.13]\label{RossThm4.13}
Let $W \geq 0$ be an integer-valued random variable such that $E(W) = \lambda > 0$ and let $W^s$ be a size-bias coupling of $W$. If $Z \sim \Poi(\lambda)$, then
\bal
d_{TV}(W,Z) \leq \min\{1,\lambda\} E|W + 1 - W^s|.
\eal
\end{theorem}

Next we follow steps (1)-(3) from Section \ref{sizebiascouplingSteinmethod} to construct the size-bias distribution $S_n^s$ with respect to $S_n$. Fix a chord diagram $\pi$. 
Pick an index $I \in [2n]$ uniformly at random, independent of $\pi$. If $X_I = 1$, set $\pi^s = \pi$. 
Otherwise, let $\pi^s$ be the chord diagram such that $\pi^s(I) = I+1$, $\pi^s(\pi(I)) = \pi(I+1)$, and $\pi^s(k) = \pi(k)$ for $k \notin \{I,I+1,\pi(I),\pi(I+1)\}$. 
That is, $\pi^s$ is the chord diagram that results by creating the simple chord $(I,I+1)$, the chord $(\pi(I),\pi(I+1))$, and fixing all other chords. 
Finally, let $S_n^s = \sum_{k \neq I} X_k^{(I)} + 1$, where $X_k^{(I)}$ is the indicator that $\pi^s$ has the simple chord $(k,k+1)$. 
It is clear that $(X_k^{(i)})_{k \neq i}$ has the distribution of $(X_k)_{k \neq i}$ conditional on $X_i = 1$, so by Proposition \ref{sizebiascouplingconstruction} we have that $(S_n,S_n^s)$ is a size-bias coupling. We record this in the following proposition. 

\begin{proposition}\label{simplechordscoupling}
Let $S_n$ be the number of simple chords in a random chord diagram of size $n$. Let $S_n^s$ be constructed as above. Then $S_n^s$ has the size-bias distribution of $S_n$. 
\end{proposition}

With the size-bias coupling established, we now prove Theorem \ref{simplechordsStein}. 

\begin{proof}[Proof of Theorem \ref{simplechordsStein}]
Let $S_n = \sum_{k=1}^{2n} X_k$ be the number of chords in a random chord diagram $C_n$ and let $Y_n$ be a $\Poi(\lambda_n)$ random variable, where $\lambda_n := ES_n = \frac{2n}{2n-1}$. 
By Proposition \ref{simplechordscoupling} and Theorem \ref{RossThm4.13}, 
\bal
d_{TV}(S_n,Y_n) &\leq \min\{1,\lambda_n\} E|S_n + 1 - S_n^s| = E\left| X_I + \sum_{k \neq I} X_k - X_k^{(I)}\right| \\
&\leq EX_I + \sum_{k \neq I} E\left| X_k - X_k^{(I)} \right|.
\eal

The first term on the right hand side is
\bal
EX_I = \sum_{i=1}^{2n} E(X_i) P(I = i) = \frac{1}{2n-1} \sum_{i=1}^{2n} \frac{1}{2n-1} = \frac{2n}{(2n-1)^2}.
\eal

For the second term, we consider two cases. If $k = i \pm 1$, then $X_k^{(i)} = 0$, so that $X_k - X_k^{(i)}$ is a zero-one random variable that equals one if and only if $X_k = 1$. This gives $E(X_i - X_i^{(i)}) = P(X_k = 1) = \frac{1}{2n-1}$. 

Otherwise, suppose $k \in [2n]\setminus \{i,i \pm 1\}$. Observe that $X_k^{(i)} \geq X_k$ by the size-bias coupling construction.
It follows that $X_k^{(i)} - X_k$ is a zero-one random variable that equals one if and only if $(i,k),(i+1,k+1) \in C_n$ or $(i,k+1),(i+1,k) \in C_n$. This event occurs with probability $\frac{2}{(2n-1)(2n-3)}$. 
Thus
\bal
\sum_{k \neq I} E\left| X_k - X_k^{(I)} \right| &= \sum_{i=1}^{2n} \frac{1}{2n-1} \sum_{k \neq i} E\left| X_k - X_k^{(i)} \right| \\
&= \sum_{i=1}^{2n} \frac{1}{2n-1} \left(\sum_{k = i \pm 1} E(X_k - X_k^{(i)}) + \sum_{k \notin \{i, i \pm 1\}} E(X_k^{(i)} - X_k)\right) \\
&= \sum_{i=1}^{2n} \frac{1}{2n-1} \left( \frac{2}{2n-1} + \frac{2(2n-3)}{(2n-1)(2n-3)} \right) \\
&= \frac{8n}{(2n-1)^2}.
\eal

Combining the above gives
\bal
d_{TV}(S_n,Y_n) &\leq \frac{2n}{(2n-1)^2} + \frac{8n}{(2n-1)^2} = \frac{10n}{(2n-1)^2} \leq \frac{10}{n}.
\eal

Finally, since $\lambda_n \to 1$ as $n \to \infty$, it follows that $d_{TV}(Y_n,Y) \to 0$. Therefore
\bal
d_{TV}(S_n,Y) &= O\left(\frac{1}{n}\right). \qedhere
\eal
\end{proof}

Theorem \ref{simplechordsStein} shows that $S_n$ converges in distribution to a $\Poi(1)$ random variable. 

%%% Simple Chords Free Chord Diagrams %%%

\subsection{Chord Diagrams with No Simple Chords}

As an application of Theorem \ref{simplechordsStein}, we obtain absolute error bounds on the number of chord diagrams of size $n$ that contain no simple chords, for all $n \geq 1$. 

\begin{proof}[Proof of Theorem \ref{simplechordfreechorddiagrams}]
Let $S_n$ be the number of simple chords in a random chord diagram of size $n$ and let $Y_n$ be a $\Poi(\lambda_n)$ random variable, where $\lambda_n = \frac{2n}{2n-1} = 1 + \frac{1}{2n-1}$. 
By definition of total variation distance, 
\bal
d_{TV}(S_n,Y_n) \geq |P(S_n = 0) - P(Y_n = 0)| = \left| \frac{s(n)}{(2n-1)!!} - e^{-\lambda_n} \right|, 
\eal
for all $n \geq 1$.
Applying Theorem \ref{simplechordsStein} gives
\bal
\left| \frac{s(n)}{(2n-1)!!} - e^{-\lambda_n} \right| \leq \frac{10}{n},
\eal
and rearranging yields
\bal
\frac{(2n-1)!!}{e}\left(e^{-\frac{1}{2n-1}} - \frac{10}{n} \right) \leq s(n) &\leq \frac{(2n-1)!!}{e}\left(e^{-\frac{1}{2n-1}} + \frac{10}{n} \right). \qedhere
\eal
\end{proof}

From Proposition \ref{simplechordfreechorddiagrams}, we get that
\bal
s(n) = \frac{(2n-1)!!}{e}(1 + o(1)),
\eal
which implies that $s(n)$ is asymptotically $\frac{(2n-1)!!}{e}$. Therefore the probability that a random chord diagram of size $n$ contains no simple chords is asymptotically $\frac{1}{e}$. 

%%%%%%%%%%%%%%%%%%%%%%%%%%%%%%%%%%%%%%%%%%%%%%%%%%%%%%%%%

\section{Final Remarks}\label{FinalRemarks}

\subsection{} A chord has {\em length $j$} if it is of the form $(i,i+j+1)$, for $i \in [2n]$, where addition is understood to be modulo $2n$. Let $L_j$ be the number of length $j$ chords in a random chord diagram of size $n$. Note that simple chords are length $0$ chords. Acan \cite{Aca17} proved that for $0 \leq j \leq n-2$, $L_j$ converges in distribution to a $\Poi(1)$, but did not provide a convergence rate. Using essentially the same argument as in the proof of Theorem \ref{simplechordsStein}, one can obtain a $O(n^{-1})$ upper bound on the total variation distance between the distribution of $L_j$ and a $\Poi(1)$ random variable. 

\subsection{} Kim \cite{Kim19} showed that the number of descents in a random matching is asymptotically normal. It should be tractable to use size-bias coupling and Stein's method to find the rate of convergence. 

\subsection{} Feray \cite{Fer18} introduced the theory of weighted dependency graphs and gave a normality criterion in this context. 
As an application, he gave another proof of the asymptotic normality of the number of crossings in a random chord diagram. 
It would be interesting to see if there is a Stein's method version for weighted dependency graphs, analogous to the one for regular dependency graphs in \cite{BR89}. 
In another direction, Feray \cite{Fer18} states that weighted dependency graphs can be used to prove a central limit theorem for the number of $k$-crossings. It may be possible to use Stein's method to obtain a rate of convergence. 

\subsection{} Chern, Diaconis, Kane, and Rhoades \cite{CDKR15} showed that the number of crossings in a uniformly chosen set partition of $[n]$ is asymptotically normal. 
Note that a chord diagram is a special case corresponding to a partition of $[2n]$ into $n$ blocks, each of size $2$. 
More recently, Feray \cite{Fer20} generalized this result and used weighted dependency graphs to prove central limit theorems for the number of occurrences of any fixed pattern in multiset permutations and set partitions. 
Is it possible to use size-bias coupling and Stein's method to obtain rates of convergence?

\subsection{} Fix $n \geq 1$, which will denote the number of vertices in a random graph. Let $\bm{d} = (d_1,\ldots,d_n)$, with $d_i \geq 1$ for all $i \in [n]$, be a degree sequence. 
That is, vertex $i$ has degree $d_i$. Consider $2m$ half-edges (so that two half-edges can be connected to form a single edge), where $2m = \sum_{i=1}^n d_i$, and perform a random matching of these half-edges. 
The resulting random graph $\CM_n(\bm{d})$ is called the {\em configuration model with degree sequence $\bm{d}$} (see \cite{Hof17}, Ch. 7). 
Note that random chord diagrams correspond to the special case where the number of vertices is $2n$ and the degree sequence consists of all $1$'s. What is the asymptotic distribution of the number of crossings in the configuration model?

%%%%%%%%%%%%%%%%%%%%%%%%%%%%%%%%%%%%%%%%%%%%%%%%%%%%%%%%%

\section*{Acknowledgements}

We thank Jason Fulman for suggesting the problem, introducing us to Stein's method, and providing many useful suggestions and references. 
We also thank Carina Betken for a helpful conversation about size-bias couplings and introducing us to the configuration model. 
Finally we thank Mehdi Talbi, Gin Park, Apoorva Shah, Greg Terlov, and Octavio Arizmendi for stimulating discussions. 

%%%%%%%%%%%%%%%%%%%%%%%%%%%%%%%%%%%%%%%%%%%%%%%%%%%%%%%%%

\Address


\begin{thebibliography}{30}

\bibitem{Aca17}
H.~Acan, 
\textit{On a uniformly random chord diagram and its intersection graph}.
Discrete Mathematics, 340 (2017), 1967-1985.

\bibitem{AP17}
H.~Acan and B.~Pittel,
\textit{Formation of a giant component in the intersection graph of a random chord diagram}.
Journal of Combinatorial Theory, Series B, 125 (2017), 33-79.

\bibitem{APRW13}
J.~E.~Anderson, R.~C.~Penner, C.~M.~Reidys, and M.~S.~Waterman,
\textit{Topological classification and enumeration of RNA structures by genus}.
Journal of Mathematical Biology, 67 (5) (2013), 1261-1278.

\bibitem{AHH19}
O.~Angel, R.~van~der~Hofstad, and C.~Holmgren,
\textit{Limit laws for self-loops and multiple edges in the configuration model}.
Annales de l'Institut Henri Poincare Probabilities et Statistiques, 55 (3) (2019), 1509-1530.

\bibitem{AD17} 
R.~Arratia and S.~DeSalvo,
\textit{Completely effective error bounds for Stirling numbers of the first and second kinds via Poisson approximation}.
Annals of Combinatorics, 21 (2017), 1-24.

\bibitem{BR89}
P.~Baldi and Y.~Rinott, 
\textit{On normal approximations of distributions in terms of dependency graphs}. 
Annals of Probability, 17 (4) (1989), 1646-1650.

\bibitem{Bet19}
C.~Betken, 
\textit{A central limit theorem for the number of isolated vertices in a preferential attachment random graph}. 
arXiv:1910.02668 (2019).

\bibitem{BR00}
B.~Bollobas and O.~Riordan, 
\textit{Linearized chord diagrams and an upper bound for Vassiliev invariants}. 
Journal of Knot Theory and Its Ramifications, 9 (2000), 847-853.

\bibitem{BR04}
B.~Bollobas and O.~Riordan, 
\textit{The diameter of a scale free random graph}. 
Combinatorica, 24 (1) (2004), 5-34.

\bibitem{CDDSY07}
W.~Y.~C.~Chen, E.~Y.~P.~Deng, R.~R.~X.~Du, R.~P.~Stanley, and C.~H.~Yan,
\textit{Crossings and nestings of matchings and partitions}.
Transactions of the American Mathematical Society, 359 (4) (2007), 1555-1575. 

\bibitem{CDKR15}
B.~Chern, P.~Diaconis, D.~M.~Kane, and R.~C.~Rhoades,
\textit{Central limit theorem for some set partition statistics}.
Advances in Applied Mathematics, 70 (2015), 92-105. 

\bibitem{CR10}
L.~H.~Y.~Chen and A.~Rollin, 
\textit{Stein couplings for normal approximation}.
arXiv:1003.6039v2 (2010). 

\bibitem{CV07}
M.~Conger and D.~Viswanath, 
\textit{Normal approximations for descents and inversions of permutations of multisets}.
Journal of Theoretical Probability, 20 (2) (2007), 309-325.

\bibitem{Fer18}
V.~Feray,
\textit{Weighted dependency graphs}.
Electronic Journal of Combinatorics, 23 (93) (2018), 1-65.

\bibitem{Fer20}
V.~Feray,
\textit{Central limit theorems for patterns in multiset permutations and set partitions}.
Annals of Applied Probability, 30 (1) (2020), 287-323.

\bibitem{FN00}
P.~Flajolet and M.~Noy, 
\textit{Analytic combinatorics of chord diagrams},
Formal Power Series and Algebraic Combinatorics, Springer (2000), 191-201.

\bibitem{GR06} 
L.~Goldstein and G.~Reinert, 
\textit{Total variation distance for Poisson subset numbers},
Annals of Combinatorics, 10 (2006), 333-341.

\bibitem{HZ86}
J.~Harer and D.~Zagier, 
\textit{The Euler characteristic of the moduli spaces of curves},
Inventiones Mathematicae, 85 (1986), 457-486.

\bibitem{He22}
J.~He, 
\textit{A central limit theorem for descents of a Mallows permutation and its inverse},
Annales de l'Institut Henri Poincar\'{e} Probabiliti\'{e}s et Statistiques, 58 (2) (2022), 667-694. 

\bibitem{Hof17}
R.~van~der~Hofstad, 
\textit{Random graphs and complex networks, Volume 1},
Cambridge Series in Statistical and Probabilistic Mathematics. 
Cambridge University Press, Cambridge, (2017). 

\bibitem{HJ15}
C.~Holmgren and S.~Janson, 
\textit{Limit laws for functions of fringe trees for binary search trees and random recursive trees},
Electronic Journal of Probability, 20 (4) (2015), 1-51. 

\bibitem{Kim19}
G.~B.~Kim, 
\textit{Distribution of descents in matchings},
Annals of Combinatorics, 23 (1) (2019), 73-87.

\bibitem{Kon93}
M.~Kontsevich, 
\textit{Vassiliev's knot invariants},
Advances in Soviet Mathematics, 16 (2) (1993), 137-150.

\bibitem{LN11}
N.~Linial and T.~Nowik, 
\textit{The expected genus of a random chord diagram},
Discrete and Computational Geometry, 45 (1) (2011), 161-180.

\bibitem{Mah20}
A.~A.~Mahmoud, 
\textit{An asymptotic expansion for the number of 2-connected chord diagrams},
arXiv:2009.12688 (2020).

\bibitem{NS06}
A.~Nica and R.~Speicher, 
\textit{Lectures on the Combinatorics of Free Probability},
Lond. Math. Soc. Lecture Notes Series, Vol. 335, Cambridge, (2006).

\bibitem{Rio75}
J.~Riordan,
\textit{The distribution of crossings of chords joining pairs of $2n$ points on a circle},
Mathematics of Computation, 29 (129) (1975), 215-222.   

\bibitem{Ros11}
N.~Ross,
\textit{Fundamentals of Stein's method},
Probability Surveys, 8 (2011), 210-293.  

\bibitem{SE78}
P.~R.~Stein and C.~J.~ Everett,
\textit{On a class of linked diagrams, II. Asymptotics},
Discrete Mathematics, 21 (1978), 309-318. 

\bibitem{Tou52}
J.~Touchard,
\textit{Sur un probleme de configurations es sur les fractions continues},
Canadian Journal of Mathematics, 4 (1952), 2-25. 

\end{thebibliography}
\end{document}